\documentclass[12pt]{amsart}

\usepackage{amsmath, amsfonts,amsthm,times,graphics}

\usepackage[active]{srcltx}

 \makeatletter
\renewcommand*\subjclass[2][2010]{%
  \def\@subjclass{#2}%
  \@ifundefined{subjclassname@#1}{%
    \ClassWarning{\@classname}{Unknown edition (#1) of Mathematics
      Subject Classification; using '2010'.}%
  }{%
    \@xp\let\@xp\subjclassname\csname subjclassname@#1\endcsname
  }%
}

 \makeatother
\newtheorem{theorem}{Theorem}[section]
\newtheorem{lemma}[theorem]{Lemma}
\newtheorem{corollary}[theorem]{Corollary}

\newtheorem{proposition}[theorem]{Proposition}

\theoremstyle{definition}
\newtheorem{definition}[theorem]{Definition}

\newtheorem{remark}[theorem]{Remark}

 \makeatletter
\renewcommand*\subjclass[2][2010]{%
  \def\@subjclass{#2}%
  \@ifundefined{subjclassname@#1}{%
    \ClassWarning{\@classname}{Unknown edition (#1) of Mathematics
      Subject Classification; using '1991'.}%
  }{%
    \@xp\let\@xp\subjclassname\csname subjclassname@#1\endcsname
  }%
}

 \makeatother

\begin{document}

\title[On some random variables involving
 Bernoulli random variable]{
On some random variables involving Bernoulli random variable}

\author{Romeo Me\v strovi\' c}
\address{Maritime Faculty Kotor, University of Montenegro, 
85330 Kotor, Montenegro} \email{romeo@ac.me}

 
\begin{abstract}
  Motivated by the recent investigations given in  \cite{ssa}
and the fact that Bernoulli probability-type models
were often used in  the study on some problems 
in theory of  compressive sensing, here we define and study 
the  complex-valued discrete random variables $\widetilde{X}_l(m,N)$  
($0\le l\le N-1$, $1\le m\le N$). Each of these random variables is defined 
as a linear combination of $N$ independent identically distributed   
$0-1$ Bernoulli random variables. We prove that for $l\not=0$,
 $\widetilde{X}_l(m,N)$ is the zero-mean random variable,
and we also determine the variance of $\widetilde{X}_l(m,N)$ and
its real and imaginary parts. Notice that $\widetilde{X}_l(m,N)$
belongs to the class of sub-Gaussian random variables that
are significant in some areas of theory of  compressive sensing.
In particular, we prove some probability estimates for the mentioned
random variables.  These estimates are used to 
establish the upper bounds of the sub-Gaussian norm of
their real and imaginary parts. We believe that our results 
should be implemented in certain applications of 
sub-Gaussian random variables  for solving some problems in compressive 
sensing of sparse signals.
   \end{abstract}  
  \maketitle

% Sec 1.

\section{Introduction and preliminary results}

In the statistical analysis for efficient 
detection of signal components when missing data samples are present
developed by LJ. Stankovi\'c, S. Stankovi\'c and M. Amin in \cite{ssa}  
(cf. \cite{sso} and \cite{sdsv}),
a cruciaal role plays a class of complex-valued  random variables 
denoted in \cite{m8} as $X_l(m,N)$  ($0\le l\le N-1$, $1\le m\le N$).
Furthermore, in \cite{m1} (also see \cite{m9}) it 
was generalized the random variable $X_l(m,N)$.
Motivated by these random variables and the known fact
that Bernoulli probability-type models, involving Bernoulli random variable,
were often used in  the study 
on some compressive sensing-type problems (see, e.g., 
the famous paper \cite{crt} by   Cand\`{e}s, Romberg and Tao 
and \cite{mbc}), 
here we define and study complex-valued  random variables 
$\widetilde{X}_l(m,N)$  ($0\le l\le N-1$, $1\le m\le N$).
 For more information on the development of  compressive sensing 
(also known as {\it compressed sensing}, 
{\it compressive sampling}, or {\it sparse recovery}), see \cite{do}, 
\cite{fr}, \cite[Chapter 10]{s1} and \cite{sdt}. 
For an excellent survey on this topic with applications
 and related references, see  \cite{sr} (also see \cite{op}).

% DEF. 1.1
   \begin{definition} 
Let $N$, $l$ and $m$ be nonnegative  integers such that $0\le l\le N-1$
and $1\le m\le N$.  Let  $B_n$ $(n=1,\ldots,N)$ be a sequence 
of independent identically distributed  {\it Bernoulli 
random variables} (binomial distributions) taking only
the values 0 and 1 with probability 0 and $m/N$, respectively,
i.e.,
   $$
B_n=\left\{ 
 \begin{array}{ll}
0 & \mathrm{with\,\,probability \,\,} 1-\frac{m}{N}\\
1 & \mathrm{with\,\,probability \,\,} \frac{m}{N}. 
 \end{array}\right.\leqno(1)
   $$
Then the discrete random variable $\widetilde{X}_l(m,N)=
\widetilde{X}_l(m,N)$ is defined as
a sum
  $$
\widetilde{X}_l(m,N)=\sum_{n=1}^N\exp\left({-\frac{2jn l\pi}{N}}
\right)B_n.\leqno(2)
   $$
\end{definition}

We see from (2) that the real and imaginary part
of $\widetilde{X}_l(m,N)$ are real-valued random variables 
$\widetilde{U}_l(m,N)$ and $\widetilde{V}_l(m,N)$, respectively defined by 
  $$
\widetilde{U}_l(m,N)=\sum_{n=1}^N\cos\left({\frac{2jn l\pi}{N}}
\right)B_n\leqno(3)
   $$
and 
$$
\widetilde{V}_l(m,N)=-\sum_{n=1}^N\sin\left({\frac{2jn l\pi}{N}}
\right)B_n.\leqno(4)
   $$
From Definition 1.1 it follows  that the range of the random variable 
$\widetilde{X}_l(m,N)$ consists of all possible $2^N-1$ sums of elements 
of (multi)set $\{e^{-j2nl\pi/N}:\, n=1,2,\ldots,N\}$.    

Notice that for $l=0$  $\widetilde{X}_l(m,N)$ becomes
   $$
\widetilde{X}_0(m,N)=\sum_{k=1}^NB_k\sim B\left(N,\frac{m}{N}\right),
   $$ 
where $B\left(N,m/N\right)$ is the  binomial distribution
with  parameters $N$ and $p=m/N$ and the probability mass function 
given by 
 $$
{\rm Prob}\left(B\left(N,\frac{m}{N}\right)=k\right)={N\choose k}
\left(\frac{m}{N}\right)^k\left(1-\frac{m}{N}\right)^{N-k}.
  $$

Notice that by {\it De Moivre-Laplace central limit theorem}, 
the distribution of  $B\left(N,m/N\right)$ is close 
to that of  normal random variable 
${\mathcal N}\left(m, \sqrt{m(N-m)/N}\right)$ for sufficiently large $N$.

If $N=2l$, where $l\ge 1$ is an integer, then $\widetilde{X}_l(m,N)$ 
becomes
   $$
\widetilde{X}_l(m,2l)=\sum_{k=1}^{2l}(-1)^kB_k=
\sum_{k=1}^lB_{2k}-\sum_{k=1}^lB_{2k-1}\sim 
B'\left(l,\frac{m}{2l}\right)-B''\left(l,\frac{m}{2l}\right),
   $$
where $B'\left(2l,m/(2l)\right)$ and $B''\left(2l,m/(2l)\right)$ 
are independent identically distributed binomial distributions with 
parameters $l$ and $p=m/N$ (i.e., two independent copies of 
$B\left(2l,m/(2l)\right)$).

% Rem. 1.2
\begin{remark}
Set $\Sigma=\{n: 1\le n\le N \,\, \mathrm{and}\,\, B_n=1\}$,
where $B_n$ is a random variable defined by (1). Then
the cardinality $|\Sigma|$ of $\Sigma$ is also random, following a 
binomial distribution whose expected value  is equal to 
  $$
\Bbb E[|\Sigma|]=N\cdot \frac{m}{N}=m.
   $$   
We point out also that for large values of  $N$ and $m\ll N$ there holds 
$|\Sigma|/m \approx 1$ with high probability. In other words, for such
values of $N$ and $m$, the number of observed frequencies is 
$\approx  m$ with high probability.  
\end{remark}

% Section 2
\section{The  results}

We start with  the following result.

% Proposition 2.1
\begin{proposition}     
Let $N$, $l$ and $m$ be positive integers such that $l\le N-1$
 and $1\le m\le N$.  Let 
$\widetilde{X}_l(m,N)=\widetilde{U}_l(m,N)+j\widetilde{V}_l(m,N)$  
be  the discrete random variable from Definition $1.1$, 
where the random variables  $\widetilde{U}_l(m,N)$
and $\widetilde{V}_l(m,N)$ are the real and imaginary part
of $\widetilde{X}_l(m,N)$, respectively. 
Then the expected values of 
$\widetilde{X}_l(m,N)$, $\widetilde{U}_l(m,N)$ and 
$\widetilde{V}_l(m,N)$  are respectively given by
  $$
\Bbb E[\widetilde{X}_l(m,N)]=
\Bbb E[ \widetilde{U}_l(m,N)]=\Bbb E[\widetilde{V}_l(m,N)]=0.\leqno(5) 
   $$
If in addition we suppose that $N\not= 2l$, then 
for the variance of $\widetilde{X}_l(m,N)$, 
$\widetilde{U}_l(m,N)$ and $\widetilde{V}_l(m,N)$, respectively there  holds
   $$
{\rm Var}[\widetilde{X}_l(m,N)]=\frac{m(N-m)}{N},\leqno(6)
   $$
and 
  $$
{\rm Var}[\widetilde{U}_l(m,N)]={\rm Var}[\widetilde{V}_l(m,N)]=
\frac{m(N-m)}{2N}.\leqno(7)
   $$
\end{proposition}

As an immediate consequence of equalities (5) and (7) of Proposition 
2.1 and the general formula ${\rm Var}[X]=\Bbb E[X^2]-(\Bbb E[X])^2$
 for every real-valued random variable $X$, 
we obtain the following result.

% Cor. 2.2
\begin{corollary} Under the assumptions and notations of Proposition $2.1$
and the additional assumption that $N\not= 2l$ there holds 
       $$
   \Bbb E[(\widetilde{U}_l(m,N))^2]=\Bbb E[(\widetilde{V}_l(m,N))^2]=
\frac{m(N-m)}{2N}.\leqno(8)
   $$
\end{corollary}

Finally, we will prove the following main result concerning 
some probability estimates for the random variable 
$\widetilde{X}_l(m,N)$ and its real and imaginary parts.

%  Theorem 2.3
 \begin{theorem} 
Let $N$, $l$ and $m$ be positive integers such that $l\le N-1$, $N\not= 2l$
 and $1\le m\le N$. Then under notations of Proposition $2.1$, the following 
probability estimates are satisfied for each nonnegative 
real number $t:$
           \begin{itemize}
\item[(i)] ${\rm Prob}\left(|\widetilde{U}_l(m,N)|\ge t\right)\le 
2\exp\left(-\frac{4t^2}{N}\right);$
   \item[(ii)] 
${\rm Prob}\left(|\widetilde{V}_l(m,N)|\ge t\right)\le 
2\exp\left(-\frac{4t^2}{N}\right);$
\item[(iii)] ${\rm Prob}\left(\left||\widetilde{X}_l(m,N)|-
\Bbb E[|\widetilde{X}_l(m,N)|]\right|\ge t\right)\le 
2\exp\left(-\frac{2t^2}{N}\right).$
  \end{itemize}
       \end{theorem}

% Remark 2.4
\begin{remark} Let us observe that the expressions on the right hand side of
inequalities (i), (ii) and (iii) of Theorem 2.3 does not depend 
on $m$ (or equivalently, they does not depend on parameter $p:=m/N$ 
of binomial distributions $B_1,B_2,\ldots,B_N$ which occur 
in the linear expressions  for random variables 
  $\widetilde{U}_l(m,N)$, $\widetilde{V}_l(m,N)$ and $\widetilde{X}_l(m,N)$.
Notice that by using Theorem 3.4.6 of \cite{rs}, it can be proved 
the following estimate for the variables $\widetilde{U}_l(m,N)$
and $\widetilde{V}_l(m,N)$ (denoted below as $X$ with $p=m/N$) 
under condition that $m<N/2$, i.e., $p<1/2$:
   $$
{\rm Prob}\left(|X|\ge t\right)\le 
\exp\left(-\frac{\ln \left(\frac{1-p}{p}\right) t^2}{N(1-2p)}\right)\quad
{\rm for\,\, each}\quad t\ge 0.\leqno(9)
   $$     
In the limit as $p\to 1/2$, the right hand side of the above equality 
becomes   $\exp\left(-2t^2/N\right)$. Therefore, 
the estimates (i) and (ii) of Theorem 2.3  together with the estimate 
(9) immediately yield the following improvement of (i) and (ii).
 \end{remark}

%  Theorem 2.5
 \begin{theorem}
Under the assumptions and notations of Theorem $2.1$, if 
$m<N/2$ then the following 
probability estimate holds for the variables $\widetilde{U}_l(m,N)$
and $\widetilde{V}_l(m,N)$ (denoted below as $X$): 
   $$
{\rm Prob}\left(|X|\ge t\right)\le 
\exp\left(-\max\left\{ \frac{4t^2}{N}-\ln 2, 
\frac{t^2\ln \frac{N-m}{m}}{N-2m}\right\} 
\right)\,\, {\rm for\,\, each}\,\, t\ge 0.\leqno(10)
   $$
 \end{theorem}
It is easy to see that if $m=o(N)$ as $N\to\infty$, then the maximum on the 
right hand side
of  (10) is attained at the second expression for all sufficently large 
values $N$ and for each $t\ge 0$. 
For example, if $m= \lfloor \sqrt{N}\rfloor$,
then the mentioned maximum is attained  at the second expression for
each $N\ge 2304$, i.e., for each $m\ge \lfloor\sqrt{2304}\rfloor= 48$. 
A computation shows that  the mentioned maximum is attained  at the 
first  expression for every pair $(N,m)$ 
with $N=\lfloor cm\rfloor$, where $c$ is any real number
in the interval $(2,47]$.

Let us recall that a real-valued random variable $X$ is 
{\it sub-Gaussian} if its 
distribution is dominated by a normal distribution. Precisely, 
a real-valued random variable $X$ is  sub-Gaussian if there holds 
     $$
\mathrm{Prob}(|X|>t)\le \exp\left(1-\frac{t^2}{C^2}\right)
\quad {\rm for \,\, all}\quad t\ge 0,
     $$
where $C>0$ is a real constant that does not depends  on $t$.  

Sub-Gaussian random variables are introduced by Kahane \cite{ka}.
Notice that  normal and all bounded random variables are sub-Gaussian, while
 exponential random variables are not. A systematiac introduction into 
sub-Gaussan random variables can be found in \cite[Lemma 5.5 in Section 5.2.3 
and Subsection 5.2.5]{ver}; here we briefly mention the basic definitions.
Notice that the {\it  Restricted Isometry Property} (RIP) introduced 
in \cite{ct1} holds 
with high probability for any matrix generated by a sub-Gaussian random 
variable (also see \cite{ct2} and \cite{rv}). In particular, 
RIP of  Bernoulli sensing matrices is studied in \cite{ra}.
 
Let us recall that  (cf. \cite[Definitions 2.5.6 and Example 2.7.13]{ver2}) 
the {\it sub-Gaussian norm} $\Vert\cdot \Vert_{\psi_2}^{,}$
which is for the sub-Gaussian  real-valued random variable $X$ is defined as
   $$
\Vert X \Vert_{\psi_2}=
\inf\{ K>0:\, \Bbb 
E\left[\exp\left(\frac{X^2}{K^2}\right)\right]\le 2\}.\leqno(11)
    $$ 

Notice that  the sub-Gausssian norm $\Vert \cdot \Vert_{\psi_2}$ given by 
(11) is a particular case of the {\it Orlicz norm} with the    
{\it Orlicz function}  $\psi_2(x)=\exp(x^2)-1$ (see \cite[Section 2]{m7}).
For more information on the Orlicz functions and 
the associated  topological vector spaces,  see 
\cite{mu} (also see see  \cite[Chapter 7]{mp} and \cite{mpl}).
In view of the above mentioned facts, a random variable 
$X$ is  sub-Gaussian if and only if
    $$
 \Bbb E\left[\exp\left(\frac{X^2}{\psi}\right)\right]\le 2
   $$
for some real constant $\psi>0$. Hence, any bounded real-valued random 
variable $X$ is sub-Gaussian, and clearly, there holds
   $$
\Vert X \Vert_{\psi_2}\le \frac{1}{\sqrt{\ln 2}}\Vert X \Vert_{\infty}\approx
 1.20112 \Vert X \Vert_{\infty},
\leqno(12)
  $$
where $\Vert \cdot \Vert_{\infty}$ is the usual  {\it supremum norm}.
Moreover, if $X$ is a centered normal random variable with variance 
$\sigma^2$, then $X$ is sub-Gaussian with 
$\Vert X \Vert_{\psi_2}\le C\sigma$, where $C$ is an absolute constant
\cite[Subsection 5.2.4]{ver}. 
  
Another definition of the sub-Gaussian norm
$\Vert X \Vert_{\psi_2}^{'}$ for the sub-Gaussian  
random variable $X$ was given in \cite[Definition 5.7]{ver} as
   $$
\Vert X \Vert_{\psi_2}^{'}=
\sup_{p\ge 1}\left( p^{-1/2} \, \Bbb (E[|X|^p])^{1/p}\right).
    $$ 
Obviously, there holds
   $$
\Vert X \Vert_{\psi_2}^{'}\le \Vert X \Vert_{\infty}.
  $$
 
Since   $\widetilde{U}_l(m,N)$ and $\widetilde{V}_l(m,N)$
are bounded random variable with 
$\Bbb E[\widetilde{U}_l(m,N)]=\Bbb E[\widetilde{V}_l(m,N)]=0$,
the first assertion of the following result is true. 

% Prop. 2.6
  \begin{proposition} 
Let $N$, $l$ and $m$ be positive integers such that $l\le N-1$
 and $1\le m\le N$. Then $\widetilde{U}_l(m,N)$ 
and $\widetilde{V}_l(m,N)$ are centered sub-gaussian 
random variables. 
Moreover, for any positive real number $K$ such that $K>\sqrt{N}/2$ 
there holds
  $$
\Bbb E\left[\frac{(\widetilde U_l(m,N))^2}{K^2} \right]\le
1+\frac{8eNK^2}{(4K^2-N)^2} 
   $$
and 
 $$
\Bbb E\left[\frac{(\widetilde V_l(m,N))^2}{K^2} \right]\le
1+\frac{8eNK^2}{(4K^2-N)^2}. 
   $$ 
   \end{proposition}
Taking into account the  definition of sub-Gaussian norm 
$\Vert\cdot \Vert_{\psi_2}$ given by (11), as an immediate 
consequence of Proposition 2.6, we obtain the following result.

% Cor 2.7
\begin{corollary}
Let $N$, $l$ and $m$ be positive integers such that $l\le N-1$
 and $1\le m\le N$. Then
   $$
\Vert \widetilde U_l(m,N) \Vert_{\psi_2}\le 
\frac{\sqrt{N}(\sqrt{2e}+\sqrt{2e+4})}{4}\approx 1.35809\sqrt{N}.
  $$
and 
   $$
\Vert \widetilde V_l(m,N) \Vert_{\psi_2}\le 
\frac{\sqrt{N}(\sqrt{2e}+\sqrt{2e+4})}{4}\approx 1.35809\sqrt{N}
  $$
\end{corollary}

Clearly, from (3) and (4) it follows that 
$\Vert \widetilde U_l(m,N)\Vert_{\infty}\le N$ and
  $\Vert \widetilde V_l(m,N)\Vert_{\infty}\le N$ there holds. This together 
with the inequality (12) yields
  $$
\Vert \widetilde U_l(m,N) \Vert_{\psi_2}\le 
\frac{N}{\sqrt{\ln 2}}\quad {\rm and}\quad
\Vert \widetilde V_l(m,N) \Vert_{\psi_2}\le \frac{N}{\sqrt{\ln 2}}. 
  $$
Notice that the estimates from Corollary 2.7
probably are   not sharp, but they  are much better 
than these given by the above two inequalities. 
     % Remark 2.8
 \begin{remark} 
Put
  $$
Q(x):=\frac{1}{\sqrt{2\pi}}\int_x^{\infty}\exp\left(-\frac{t^2}{2}\right)\,
{\rm d}t\quad {\rm for\,\, all}\quad x\in\Bbb R^+\setminus\{0\},
  $$
i.e., $Q(x)$ is the complementary standard Gaussian cumulative density
function (also known as the $Q$-function). Then for 
$Q(x)$ the following exponential upper and lower bounds holds 
(see, e.g., \cite[Section 3.3]{ve}):
  $$   
\frac{1}{\sqrt{2\pi}}\cdot\frac{x}{1+x^2}\cdot
\exp\left(-\frac{x^2}{2}\right)<Q(x)<\frac{1}{\sqrt{2\pi}x}\cdot 
\exp\left(-\frac{x^2}{2}\right)\quad {\rm for\,\, all}\quad x\in\Bbb R.
  $$
The left hand side of the above double inequality allows us 
to replace the right hand side of inequalities (i) and  (ii) 
of Theorem 2.3  with the expression involving the function 
$Q\left(2t\sqrt{2}/\sqrt{N}\right)$; namely, if 
$X=|\widetilde{U}_l(m,N)|$ or $Y=|\widetilde{V}_l(m,N)|$, then 
      $$
{\rm Prob}\left(X\ge t\right)\le\frac{(N+8t^2)\sqrt{\pi}}{t\sqrt{N}}Q
\left(\frac{2t\sqrt{2}}{\sqrt{N}}\right).
       $$  
    \end{remark} 

  % Remark 2.9
\begin{remark}
From (6) and \cite[(19) of Theorem 2.4]{m8} it follows that 
   $$
\frac{{\rm Var}[X_l(m,N)]}{{\rm Var}[\widetilde{X}_l(m,N)]}
=\frac{N}{N-1}.\leqno(13)
   $$
For related discussion on the proprtion (13), see \cite[Remark 1.2]{m7}.
 \end{remark}

Proofs of Proposition 2.1, Theorem 2.3 and Proposition 2.6
 are given in Section 3. 
 
% Sec. 3
\section{Proofs of the results}

For the proof of Proposition 2.1, we will need the 
following result.

% Lemma 3.1
\begin{lemma} $($\cite[(38) of  Lemma 4.1]{m8}$)$. Let $N$ and  $l$ be  positive 
integers such that $l\le N-1$ and $N\not= 2l$, then
   $$
\sum_{k=1}^N\cos\frac{4kl\pi}{N}=\sum_{k=1}^N\sin\frac{4kl\pi}{N}  =0.  
\leqno(14)
   $$ 
\end{lemma}

\begin{proof}[Proof of Proposition  $2.1$] By the equality (2) of Definition 
1.1 and the obvious  fact that $\Bbb E[B_n]=m/N$ for each $n=1,\ldots,N$,
and using the linearity property of the expectation, 
we find that for each $l\not=0$ with $l\le N-1$,
  \begin{equation*}\begin{split}
\Bbb E\left[\widetilde{X}_l(m,N))\right]&=\Bbb E\left[\sum_{n=1}^NB_n\exp
\left({-\frac{2jn l\pi}{N}}\right)\right]=
\sum_{n=1}^N\exp\left({-\frac{2jn l\pi}{N}}\right)\Bbb E[B_n]\\
&=\frac{m}{N}\sum_{n=1}^N\exp\left({-\frac{2jn l\pi}{N}}\right)=
\frac{m}{N}\cdot \sum_{n=0}^{N-1}\exp\left({-\frac{2jn l\pi}{N}}\right) \\
&= \frac{m}{N}\cdot\exp\left(-\frac{2j l\pi}{N}\right)\cdot\frac{\exp\left({-2jl \pi}\right)-1}
{\exp\left(-\frac{2j l\pi}{N}\right)-1}=0.
  \end{split}\end{equation*}
The above equality and the expression
$\Bbb E[\widetilde{X}_l(m,N)]=\Bbb E[\widetilde{U}_l(m,N)]+
j\Bbb E[ \widetilde{V}_l(m,N)]$ imply all the equalities in (5).
 
Further, using the additive property  for  the variance of 
 a finite  sum  of {\it uncorrelated random variables} (and hence, for a finite sum 
of independent identically distributed random variables), 
for all $l=0,1,\ldots ,N-1$ we obtain
   \begin{equation*}\begin{split}
{\rm Var}[\widetilde{X}_l(m,N)]&= 
\sum_{n=1}^N{\rm Var}\left[\exp\left({-\frac{2jn l\pi}{N}}\right)B_n
\right]\\
&= \sum_{n=1}^N\left(\left|\exp\left({-\frac{2jn l\pi}{N}}
\right)\right|^2{\rm Var}[B_n]\right)\\
\qquad\qquad\qquad &=\sum_{n=1}^N{\rm Var}[B_n]=N\cdot {\rm Var} [B_1]=
N\cdot \left( \Bbb E[B_1^2]-(\Bbb E[B_1])^2\right)\qquad\qquad\qquad\qquad\\
&=N\cdot\left(\frac{m}{N}-\left(\frac{m}{N}\right)^2\right)=\frac{m(N-m)}{N},
  \end{split}\end{equation*}
which implies the first expression in (5). Similarly, by using the expansion
   $$
\widetilde{U}_l(m,N)=\sum_{n=1}^N\cos\frac{ 2nl\pi}{N}B_n,
   $$ 
the trigonometric identity $\cos^2 \alpha=(1+\cos 2 \alpha)/2$
and the first identity of (14) from Lemma 3.1, we obtain   
     \begin{equation*}\begin{split}
{\rm Var}[\widetilde{U}_l(m,N)]&=
 \sum_{n=1}^N{\rm Var}\left[\cos\frac{ 2nl\pi}{N}B_n
\right]= \sum_{n=1}^N\cos^2\frac{ 2nl\pi}{N}{\rm Var} [B_n]\\
&=\sum_{n=1}^N\cos^2\frac{ 2nl\pi}{N}\left
(\frac{m}{N}-\left(\frac{m}{N}\right)^2\right)=\frac{m(N-m)}{N^2}
\sum_{n=1}^N\cos\frac{ 2nl\pi}{N}\\
&=\frac{m(N-m)}{N^2}\sum_{n=1}^N\frac{1+\cos \frac{4nl\pi}{N}}{2}=
\frac{m(N-m)}{N^2}\left(\frac{N}{2}+\frac{1}{2}
\sum_{n=1}^N\cos\frac{4nl\pi}{N}\right)\\
&=\frac{m(N-m)}{2N},
 \end{split}\end{equation*}
whence it follows the equality  (6).
 
Proceeding analogously as above, by using the expansion
   $$
\widetilde{V}_l(m,N)=-\sum_{n=1}^N\sin\left(\frac{ 2nl\pi}{N}\right)B_n,
   $$ 
the trigonometric identity $\sin^2 \alpha=(1-\cos 2 \alpha)/2$
and the second identity of (14) from Lemma 3.1, we obtain
 the expression (7).
This completes proof of Proposition 2.1.
   \end{proof}

Proof of Theorem 2.3 is based on the following 
{\it McDiarmid's inequality}, also known as the {\it bounded-difference 
inequality} given in  \cite[Theorem 3.1]{di1} (also see \cite{di2}).

% Theorem 3.2
\begin{theorem} $($McDiarmid's inequality$)$. 
Let $X_1,\ldots,X_N$ be independent 
(not necessarily identically distributed) real-valued random variables all 
taking values in a measurable set $\chi$. Further, 
let $f:\chi^N\mapsto \Bbb R$
be a measurable function   of $X_1,\ldots,X_N$ (random variable) 
that satisfies the inequality
(the bounded difference assumption)
    $$
\left| f(x_1,\ldots,x_k,\ldots,x_N)-f(x_1,\ldots,x_k^{'},\ldots,x_N)\right|
\le c_k\leqno(15)
  $$
for all $x_1,\ldots, x_N,x_k^{'}\in\chi$ with any fixed  
$k\in\{1,\ldots,N\}$, where 
$c_k$  $(k=1,\ldots,N)$ are arbitrary nonnegative real constants. 
Consider a random variable $X=f(X_1,\ldots,X_N)$.
Then for every real number $t\ge 0$,
   $$
{\rm Prob}\left(|X-\Bbb E[X]|\ge t\right)\le 
2\exp\left(-\frac{2t^2}{\sum_{k=1}^N c_k^2}\right).\leqno(16) 
  $$ 
\end{theorem}
  
\begin{proof}[Proof of Theorem $2.3$]
 First notice that by definition,
  $$
\widetilde{U}_l(m,N)=\sum_{k=1}^N\cos\frac{ 2kl\pi}{N}X_k,
   $$
where $X_k=B_k$ $(k=1,\ldots,N)$ are independent 
{\it Bernoulli} 0-1 {\it random variables}. Accordingly, if we define 
the real-valued function $f:\{0,1\}^N\mapsto\Bbb R$ defined on 
two-points set $\{0,1\}$  as 
   $$
f(x_1,\ldots ,x_N)=\sum_{k=1}^N\cos\frac{ 2kl\pi}{N}x_k,\quad 
{\rm for\,\, all}\quad 
x_1,\ldots,x_N\in \{0,1\},\leqno(17)
   $$  
then we can write 
   $$
\widetilde{U}_l(m,N)=f(X_1,\ldots ,X_N).
   $$
Substituting the expression for the function $f$ given by (17) into
the inequality (15), it becomes
     $$
 \left|\cos\frac{2kl\pi}{N}x_k -\cos\frac{2kl\pi}{N}x_k^{'}\right|=
|x_k-x_k^{'}|\cdot \left|\cos\frac{2kl\pi}{N}\right|\le c_k,\leqno(18) 
    $$
 for all $x_k, x_k^{'}\in \{0,1\}$  with $k=1,\ldots,N$.
   
Clearly, for any fixed $k\in\{1,\ldots,N\}$ the inequality (18)
holds with $c_k=\left|\cos\frac{ 2kl\pi}{N}\right|$.
 
Then by using the identity $\cos^2\alpha=(1+\cos 2\alpha)/2$ and the first
identity of (14) from Lemma 3.1, we obtain 
     \begin{equation*}\begin{split}
\sum_{k=1}^N c_k^2&=\sum_{k=1}^N\cos^2\frac{ 2kl\pi}{N}=
\sum_{k=1}^N\frac{1+ \cos\frac{4kl\pi}{N}}{2}=\\
&= \frac{N}{2}+\frac{1}{2}\sum_{k=1}^N\cos\frac{4kl\pi}{N}=
\frac{N}{2}.
   \end{split}\end{equation*}
Substituting the above expression $\sum_{k=1}^N c_k^2=N/2$ into 
(16), we obtain that for every $t\ge 0$,
     $$
{\rm Prob}\left(|\widetilde{U}_l(m,N)|\ge t\right)\le 
2\exp\left(-\frac{4t^2}{N}\right).
    $$
This proves the inequality (i) of Theorem 2.3.
   
As by the identity $\sin^2\alpha=(1-\cos 2\alpha)/2$ and the second
identity of (14) from Lemma 3.1, 
 \begin{equation*}\begin{split}
\sum_{k=1}^N\sin^2\frac{ 2kl\pi}{N}&=
\sum_{k=1}^N\frac{1- \cos\frac{4kl\pi}{N}}{2}=\\
&= \frac{N}{2}-\frac{1}{2}\sum_{k=1}^N\cos\frac{4kl\pi}{N}=
\frac{N}{2}
   \end{split}\end{equation*}
and   
  $$
\widetilde{V}_l(m,N)=\sum_{k=1}^N\sin\frac{ 2kl\pi}{N}B_k,
   $$
proof of the inequality (ii) of Theorem 2.3 can be deduced 
in the same manner as that of (i),  and hence, may be omitted.

Finally, in order to prove the inequality (iii) of Theorem 2.3,
notice that $|\widetilde{X}_l(m,N)|$ is a real-valued random variable
such that 
   $$
\left|\widetilde{X}_l(m,N)\right|=
\left|\sum_{k=1}^N\exp\left({-\frac{2jk l\pi}{N}}\right)B_k
\right|.\leqno(19)
   $$
Similarly as above, take $X_k=B_k$  ($k=1,\ldots, N$)
and  define the real-valued function $g:\{0,1\}^N\mapsto\Bbb R$ as 
   $$
g(x_1,\ldots ,x_N)=\left|
\sum_{k=1}^N\exp\left({-\frac{2jk l\pi}{N}}\right)x_k\right|,\quad 
{\rm for\,\, all}\quad 
x_1,\ldots,x_N\in \{0,1\}.\leqno(20)
   $$
Substituting the expression for the function $f$ given by (20) into
the inequality (15), it becomes
    \begin{equation*}\begin{split}
&\left|\exp\left({-\frac{2jk l\pi}{N}}\right)x_k-
 \exp\left({-\frac{2jk l\pi}{N}}\right)x_k^{'}\right|\\
(21)\qquad\qquad\qquad 
&=|x_k- x_k^{'}|\cdot \left|\exp\left({-\frac{2jk l\pi}{N}}\right)\right|
=|x_k- x_k^{'}| \le c_k,\qquad\qquad\qquad
        \end{split}\end{equation*}
 for all $x_k, x_k^{'}\in \{0,1\}$  with $k=1,\ldots,N$.
    Obviously, for any fixed $k\in\{1,\ldots,N\}$ the inequality (21)
holds with $c_k=1$. Finally, substituting 
 the value $\sum_{k=1}^N c_k^2=N$ into 
(16), we obtain that for every $t\ge 0$,
     $$
  {\rm Prob}\left(\left||\widetilde{X}_l(m,N)|-
\Bbb E[|\widetilde{X}_l(m,N)|]\right|\ge t\right)\le 
2\exp\left(-\frac{2t^2}{N}\right).
    $$
Therefore, the estimate (iii) holds and proof of the theorem is completed. 
    \end{proof}

  \begin{proof}[Proof of Proposition $2.6$]
Since the estimates (i) and (ii) of Theorem 2.3 are the same,
it is sufficient  to prove the first inequality of Proposition 2.6.
We follow the part of proof of Proposition 2.5.2 
in \cite[Subsection 2.5.1]{ver2}. 
 It is known (see, e.g., \cite[Appendix A, p. 434]{ti}) that for any 
nonnegative real-valued  random variable
$X$ with probability distribution function $F(x):={\rm Prob}(X< x)$
($x\in\Bbb R$),
   $$
\Bbb E[X]=\int_0^{\infty}(1-F(v))\,{\rm d}v,
    $$
which can be written as 
     $$
\Bbb E[X]=\int_0^{\infty}{\rm Prob}(X\ge v)\,{\rm d}v.
    $$
Applying the above equality for $X=\left|\widetilde U_l(m,N)\right|^n$,
after the change of variables $v=t^{2n}$, ${\rm d} v=2nt^{2n-1}{\rm d} t$,
 using the estimate (i) of 
Theorem 2.3 and the well known inequality 
  $\Gamma (x)\le x^xe^{1-x}$
($x\ge 1$, where $\Gamma (x)=\int_{0}^{\infty}s^{x-1} e^{-s}\,{\rm d}s$ 
for $x>0$, is the {\it complete gamma function}),
for every integer $n\ge 1$ we obtain   
   \begin{equation*}\begin{split}
\Bbb E\left[\widetilde U_l(m,N)^{2n}\right] 
&=\int_0^{\infty}{\rm Prob}(\left|\widetilde U_l(m,N)\right|\ge t)
2nt^{2n-1}\,{\rm d}t\le 4n\int_0^{\infty}\exp\left(-\frac{4t^2}{N}\right)
t^{2n-1}\,{\rm d}t\\
&=({\rm the\,\, change\,\, of\,\, variables}\,\, 
\frac{4t^2}{N}=s, {\rm d}t=\frac{\sqrt{N} {\rm d}s}{4\sqrt{s}})\\
&=\frac{nN^n}{2^{2n-1}}\int_0^{\infty} e^{-s}s^{n-1}\,{\rm d}s
=\frac{nN^n}{2^{2n-1}}\Gamma(n)\le 
\frac{nN^n}{2^{2n-1}} \cdot n^ne^{1-n}\\
&= \frac{n^{n+1}N^n}{2^{2n-1}e^{n-1}}.
   \end{split}\end{equation*}
By using the Taylor expansion of 
$\widetilde U_l(m,N)^2/K^2$ ($K$ is a positive real 
constant), the linearity of the expected value,  the above inequality
and the well known inequality $n!\ge (n/e)^n$ ($n\in\Bbb N$),
we find that  
  \begin{equation*}\begin{split}
\Bbb E\left[\frac{\widetilde U_l(m,N)^2}{K^2} \right]
&=1+\sum_{n=1}^{\infty}
\frac{\Bbb E\left[\widetilde U_l(m,N)^{2n}\right]}{K^{2n}n!}\\
(22)\qquad\qquad\qquad & \le 1+
\sum_{n=1}^{\infty}\frac{n^{n+1}N^n}{2^{2n-1}e^{n-1}K^{2n}n!}
\le 1+\sum_{n=1}^{\infty}\frac{n^{n+1}N^n}{2^{2n-1}e^{n-1}K^{2n}}
\cdot{e^n}{n^n}\qquad \qquad\qquad\qquad \\
&= 1+2e\sum_{n=1}^{\infty}n\left(\frac{N}{4K^2}\right)^n.
     \end{split}\end{equation*}
Since for $|x|<1$ we have 
      \begin{equation*}\begin{split}
\sum_{n=1}^{\infty}nx^n&=x\sum_{n=1}^{\infty}nx^{n-1}=
\frac{{\rm d}}{{\rm d}x}\sum_{n=1}^{\infty}x^n=
\frac{{\rm d}}{{\rm d}x}\frac{x}{1-x}\\
&= \frac{x}{(1-x)^2},
     \end{split}\end{equation*}
substituting the above identity with $N/(4K^2)=x$  
(which is by the assumption $<1$)  into (22), we get 
   $$
\Bbb E\left[\frac{\widetilde U_l(m,N)^2}{K^2} \right]\le
1+\frac{8eNK^2}{(4K^2-N)^2}, 
   $$
as asserted.
  \end{proof}

   \end{document}